\documentclass{amsart}
\usepackage[english]{babel}
\usepackage[T1]{fontenc}
\usepackage[color,matrix,arrow]{xy}
\usepackage{tikz}
\usetikzlibrary{decorations.markings,arrows,automata,backgrounds}
\usepackage{amsmath,amssymb,amsthm,amsfonts,amscd,mathrsfs,tikz-cd}
\usepackage{listings}
\usepackage{color}
\tikzset{->-/.style={decoration={markings, mark=at position 0.5 with 
    {\arrow{triangle 60}}}, postaction={decorate}}}
\usepackage{algpseudocode}
\usepackage{comment}
\usepackage{booktabs}
\usetikzlibrary{calc}
\usepackage{graphicx}
\usepackage[colorinlistoftodos]{todonotes}
\usepackage[colorlinks=true, allcolors=blue]{hyperref}
\usepackage{float}
\usepackage{mathtools}
\usetikzlibrary{shapes.geometric}

\theoremstyle{definition}

\newcommand{\st}{\mathrm{st}}
\newcommand{\lk}{\mathrm{lk}}
\newcommand{\ZZ}{\mathbb{Z}}
\newcommand{\M}{\mathcal{M}}
\newcommand{\GM}{\mathcal{GM}}
\newcommand{\HH}{\mathcal{H}}

\newtheorem{thm}{Theorem}

\newtheorem{lem}[thm]   {Lemma}
\newtheorem{cor}[thm]   {Corollary}
\newtheorem{rem}[thm]   {Remark}
\newtheorem{defn}[thm]  {Definition}

\newtheorem{prop}[thm]  {Proposition}
\newtheorem{conj}[thm]  {Conjecture}

\newcounter{foo}  \Alph{foo}

\newenvironment{example}
{\medskip\par\noindent{\sc Example}\ }
{\par}

\title[discrete Morse matchings of the $n$-simplex]{The complex of discrete Morse matchings of the $n$-simplex: homotopy types and structural results}
\author[N.A. Scoville]{Nicholas A. Scoville}

\address{Department of Mathematics and Computer Science, Ursinus College, Collegeville PA 19426}

\email{nscoville@ursinus.edu}


\begin{document}

\begin{abstract}
The complex of discrete Morse matchings $\M(K)$, introduced by Chari and  Joswig, is a simplicial complex whose simplices are the acyclic matchings  on the Hasse diagram of $K$. Its homotopy type is known in only a handful of cases. In this paper, we compute the homotopy types of $\M(\Delta^3)$ and $\M(\partial\Delta^3)$, the corresponding pure complexes $\M_{P}(\Delta^3) \simeq \M_{P}(\partial\Delta^3)$, and the generalized complex of discrete Morse matchings $\GM(\Delta^3) \simeq \GM(\partial\Delta^3)$. For general $n$ we prove the identity  $f(n) = (n+1) \cdot |\text{top-dimensional facets of } \M(\Delta^n_{(n-2)})|$, reducing the enumeration of optimal matchings on $\Delta^n$ to an enumeration on its $(n-2)$-skeleton, and we show that the inclusion $\M(K) \hookrightarrow \M(CK)$ is null-homotopic for any cone. We also compute the $f$-vector of $\M(\Delta^4)$, whose top entry $f(4) = 380{,}125$ is the number of optimal discrete Morse matchings on $\Delta^4$. We conclude with two conjectures extending the $\M_{P}$ and $\GM$ equivalences to all $n$.
\end{abstract}

\maketitle

\section{Introduction}

Let $K$ be a finite simplicial complex. The complex of discrete Morse matchings\footnote{In the literature, this has been called both the Morse complex and the complex of discrete Morse functions. The former name causes confusion with the name of the chain complex induced by the critical simplices of a discrete Morse function on $K$, while in the latter case, we agree with J.\ Br\"uggemann who points out that this construction refers ``to a complex whose simplices correspond to [discrete] Morse matchings rather than [discrete] Morse functions'' \cite[p. 26]{bruggemann2026}.} $\M(K)$ was introduced by Chari and Joswig \cite{CJ-2005}. We may view $\M(K)$ as a parameter space of gradient vector fields on $K$, or equivalently of discrete Morse matchings or acyclic matchings in the sense of \cite{F-95,F-02,CJ-2005}. While not every matching yields a substantive reduction, each simplex of $\M(K)$ encodes a consistent set of cancellations and hence a choice of simplifying $K$ up to simple homotopy. The topology of $\M(K)$ reflects how flexible these choices are. For example, $i$-connectivity measures the extent to which families of matchings can be deformed through matchings, as studied in \cite{scoville2020higher}, in analogy with Cerf-theoretic phenomena for spaces of smooth Morse functions \cite{cerf1970,bruggemann2026}. Simplices of maximum dimension correspond to discrete Morse matchings that minimize the total number of critical simplices, and we call such matchings \textbf{optimal}. In favorable cases, one can achieve an optimal matching with the fewest critical simplices allowed by the discrete Morse inequalities \cite[Corollary 3.7]{F-95}. Optimal and near-optimal matchings have important computational applications since discrete Morse reductions can substantially reduce complexes prior to homology or persistent homology computations in applied topology and topological data analysis \cite{Bauer2021,robins2011}.

Despite the central role of discrete Morse theory in combinatorial topology, the homotopy type of $\M(K)$ is known only in a small number of cases. Kozlov computed the homotopy type of $\M(K)$ for paths, cycles, and the complete graphs in \cite{Kozlov99}; the homotopy type has been determined for several other families of graphs and special cases in \cite{donovan_lin_scoville_2022,DonovanScoville2023}. The automorphism group of $\M(K)$ has been computed in terms of the automorphism group of $K$ in \cite{LinSco19}, and Capitelli and Minian have shown that $\M(K)$ encodes the isomorphism type of $K$ in \cite{CM-17}.

A natural next test case is the $n$-simplex $\Delta^n$ and its boundary $\partial\Delta^n$. The homotopy type of $\M(\Delta^2)$ is known from \cite{CJ-2005}, and even here very little is determined for $n \geq 3$. The principal aim of this paper is to determine the homotopy types of $\M(\Delta^3)$, $\M(\partial\Delta^3)$, and their associated pure and generalized complexes. Along the way we develop structural tools that apply for all $n$ and present computational data for $n = 4$.

Our main results are for $n = 3$. Combining direct homology computations with the framework of Bravo and Camarena in \cite{BravoCamarena24}, we show in Section \ref{sec:n3} that
$$
\M(\Delta^3) \simeq \bigvee^{99} S^4,
\qquad
\M(\partial\Delta^3) \simeq \bigvee^{21} S^3 \vee \bigvee^{24} S^4,
$$
and that the pure complexes in both cases share a single homotopy type in $\M_{P}(\Delta^3) \simeq \M_{P}(\partial\Delta^3) \simeq \bigvee^{81} S^3$. These appear as Theorems \ref{thm: h.t. of delta3}, \ref{thm: h.t. of partial delta3}, and \ref{thm: h.t. of pure}, respectively. 

In addition, we develop structural results that apply for all $n$. First, we construct an explicit bijection between the top-dimensional facets of $\M(\Delta^n)$ and those of $\M(\partial\Delta^n)$ in Proposition \ref{prop: top dim bijection}. This implies that the two complexes admit the same number of optimal matchings. Second, we record a ``layer count'' for optimal matchings in Lemma \ref{lem: counting pairs} and use it to reduce their enumeration to an enumeration on the $(n-2)$-skeleton. This defines a surjection from optimal matchings on $\Delta^n$ to optimal matchings on $\Delta^n_{(n-2)}$ with uniform fiber size $n+1$, yielding the identity
$$
f(n) = (n+1) \cdot \bigl|\text{top-dimensional facets of } \M(\Delta^n_{(n-2)})\bigr|
$$
where $f(n)$ is the number of optimal matchings on $\Delta^n$ (Theorem \ref{thm: morse bijection}). Third, we prove in Theorem \ref{thm: null-homotopy} that for any finite simplicial complex $K$, the inclusion $\M(K) \hookrightarrow \M(CK)$ into the complex of discrete Morse matchings of the cone is null-homotopic. The proof involves a one-step contiguity using a single explicit Hasse edge of the cone, and when applied to the decomposition $\Delta^n = v * \Delta^{n-1}$ it splits the long exact sequence of the pair $(\M(\Delta^n), \M(\Delta^{n-1}))$.

For $n = 4$ we present partial computational evidence. Exhaustive parallel enumeration gives the full $f$-vector of $\M(\Delta^4)$; its top entry $f(4) = 380{,}125$ agrees with the prediction of our restriction theorem via the identity $f(4) = 5 \cdot 76{,}025$. The reduced Euler characteristic $\widetilde{\chi}(\M(\Delta^4)) = 212{,}456$ is positive, which suffices to rule out any homotopy equivalence of $\M(\Delta^4)$ with a wedge of spheres in a single odd dimension. The full integral homology of $\M(\Delta^4)$ remains out of reach, as the largest boundary matrix has dimensions on the order of $10^9$, beyond what our present computational setup can handle.

We close the paper with a comparison to the generalized complex of discrete Morse matchings $\GM(K)$ of \cite{scoville2020higher}, which drops the acyclicity condition and is the simplicial complex of all matchings on the Hasse diagram of $K$. Unlike for $\M$, where $\M(\Delta^3) \not\simeq \M(\partial\Delta^3)$, we show in Section \ref{sec: GM} that $\GM(\Delta^3) \simeq \GM(\partial\Delta^3) \simeq \bigvee^{39} S^4$, and we conjecture that $\GM(\Delta^n) \simeq \GM(\partial\Delta^n)$ for all $n \geq 2$ in Conjecture \ref{conj:GM}. Together with the conjectured pure complex equivalence $\M_{P}(\Delta^n) \simeq \M_{P}(\partial\Delta^n)$, this raises the question of precisely which additional conditions on a matching cause the distinction between a simplex and its boundary to collapse at the level of homotopy type.

All computational claims in this paper have been independently verified by the software accompanying it.\footnote{Available at
\url{https://github.com/nscoville/morse-complex-verification}.}

\section{Background}

Let $K$ be an abstract, finite simplicial complex and let $K_{(i)}$ denote the $i$-skeleton of $K$. 

\subsection{Discrete Morse theory}

Our reference for general background in discrete Morse theory is \cite{KnudsonBook, KozlovDMTbook, F-02, scoville19}.  

\begin{defn}
Let $K$ be a simplicial complex, and let  $\sigma^{(p)}$ denote a simplex of dimension $p$.  A \textbf{discrete vector field} $V$ on $K$ is any non-empty set $V$ consisting of pairs of simplices $(\sigma^{(p)}, \tau^{(p+1)})$ with $\sigma^{(p)}\subseteq \tau^{(p+1)}$ where each simplex of $K$ is in at most one pair. Any pair in $(\sigma,\tau)\in V$ is called a \textbf{regular pair}. Any simplex in $K$ which is not in $V$ is called \textbf{critical}. We denote by $m_i^V=m_i$ the number of critical simplices of $V$ of dimension $i$. 
\end{defn}

\begin{defn}\label{defn: gvf}
Let $V$ be a discrete vector field on a simplicial complex $K$.  A \textbf{$V$-path} or \textbf{gradient path}  is a sequence of simplices 
$$
\alpha^{(p)}_0, \beta^{(p+1)}_0, \alpha^{(p)}_1, \beta^{(p+1)}_1, \alpha^{(p)}_2\ldots , \beta^{(p+1)}_{k-1}, \alpha^{(p)}_{k}
$$ 
of $K$ such that $(\alpha^{(p)}_i,\beta^{(p+1)}_i)\in V$ and $\beta^{(p+1)}_i>\alpha_{i+1}^{(p)}\neq \alpha_{i}^{(p)}$ for $0\leq i\leq k-1$. If $k\neq 0$, then the $V$-path is called  \textbf{non-trivial.}  A $V$-path is said to be  \textbf{closed} if $\alpha_{k}^{(p)}=\alpha_0^{(p)}$.  A discrete vector field $V$ which contains no  non-trivial closed $V$-paths is called a \textbf{gradient vector field}. 
\end{defn}
    
If the gradient vector field consists of only a single element, we say it is a \textbf{primitive} gradient vector field.

\subsection{The complex of discrete Morse matchings}

We now introduce the central object of our study. There are two equivalent formulations which we will use interchangeably throughout the paper.

\begin{defn}\label{MorseComplexDef2}
The \textbf{complex of discrete Morse matchings} of $K$, denoted $\M(K)$, is the simplicial complex whose vertices are the non-empty primitive gradient vector fields on $K$ and whose $k$-simplices are the gradient vector fields on $K$ consisting of $k+1$ regular pairs. A gradient vector field $f = \{f_0, \ldots, f_k\}$ is identified with the simplex spanned by its primitive subfields $f_i \leq f$, $0 \leq i \leq k$.

The \textbf{complex of pure discrete Morse matchings} of $K$, denoted $\M_{P}(K)$, is the subcomplex of $\M(K)$ generated by the non-empty
gradient vector fields of maximum cardinality; equivalently, it is the subcomplex generated by the simplices of $\M(K)$ of top dimension.
\end{defn}

\begin{rem}
We require gradient vector fields to be non-empty in Definition \ref{MorseComplexDef2}. The reason we do so is that the empty gradient vector field is compatible with every other gradient vector field, so admitting it would make $\M(K)$ a cone on itself and force $\M(K)$ to be contractible for every $K$.
\end{rem}

An equivalent, more combinatorial description proceeds via the Hasse diagram. The \textbf{Hasse diagram} of $K$, denoted $\HH(K)$, is the directed graph whose vertices are the simplices of $K$, with a directed edge $\sigma \to \tau$ whenever $\sigma \subsetneq \tau$ and $\dim(\tau) = \dim(\sigma) + 1$. A \textbf{matching} on $\HH(K)$ is a set $W$ of edges of $\HH(K)$ no two of which share a vertex. Given such a matching $W$, let $\HH_W(K)$ denote the directed graph obtained from $\HH(K)$ by reversing the orientation of every edge in $W$. The matching $W$ is called an \textbf{acyclic matching} or a \textbf{discrete Morse matching} if $\HH_W(K)$ contains no directed cycle, and $\HH_W(K)$ is the \textbf{modified Hasse diagram.}

To connect the two perspectives, an edge $\sigma \to \tau$ of $\HH(K)$ corresponds to a regular pair $(\sigma, \tau)$, a matching on $\HH(K)$ corresponds to a discrete vector field, and the acyclicity condition on $\HH_W(K)$ is exactly the absence of non-trivial closed $V$-paths, i.e., a gradient vector field. In particular, $\M(K)$ is equivalently described as the simplicial complex whose vertices are the edges of $\HH(K)$ and whose simplices are the acyclic matchings on $\HH(K)$. We will move freely between ``gradient vector field'' and ``acyclic matching,'' and between ``regular pair $(\sigma, \tau)$'' and ``edge of Hasse diagram $\sigma \to \tau$'' and ``optimal discrete Morse matching'' and ``maximum acyclic matching'' according to which formulation is most convenient.

\begin{example}\label{ex: delta2}
\begin{enumerate}
\item[(a)] We compute $\M(\partial\Delta^2)$ by direct construction. With $a, b, c$ the vertices of $\partial\Delta^2$, one checks that $\M(\partial\Delta^2)$ is the $1$-dimensional simplicial complex shown below:

$$
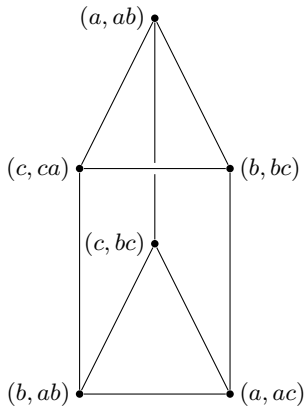
\begin{tikzpicture}[scale=1]
\node[inner sep=1pt, circle, fill=black](a) at (0,2) {};
\node[inner sep=1pt, circle, fill=black](b) at (1,0) {};
\node[inner sep=1pt, circle, fill=black](c) at (-1,0) {};
\node[inner sep=1pt, circle, fill=black](d) at (0,-1) {};
\node[inner sep=1pt, circle, fill=black](e) at (1,-3) {};
\node[inner sep=1pt, circle, fill=black](f) at (-1,-3) {};

\draw[-]  (a)--(b);
\draw[-]  (a)--(c);
\draw[-]  (a)--(d);
\draw [white, line width=1.5mm] (b) -- (c) node[midway, left] {};
\draw[-]  (b)--(c);

\draw[-]  (d)--(e);
\draw[-]  (d)--(f);
\draw[-]  (e)--(f);

\draw[-]  (b)--(e);
\draw[-]  (c)--(f);

\node[anchor = east]  at (a) {\small{$(a,ab)$}};
\node[anchor =  west]  at (b) {\small{$(b,bc)$}};
\node[anchor = east]  at (c) {\small{$(c,ca)$}};
\node[anchor = east]  at (d) {\small{$(c,bc)$}};
\node[anchor =  west]  at (e) {\small{$(a,ac)$}};
\node[anchor = east]  at (f) {\small{$(b,ab)$}};
\end{tikzpicture}
$$
This is the $1$-skeleton of a triangular prism, a graph on $6$ vertices with $9$ edges. Its first Betti number is $b_1 = 1 - \chi = 1 - (6 - 9) = 4$, so $\M(\partial\Delta^2) \simeq \bigvee^4 S^1$.

\item[(b)] Chari and Joswig computed the homotopy type of $\M(\Delta^2)$ in \cite[Proposition 5.1]{CJ-2005}. We give an alternative argument here. We will compute the homotopy type of $\M(\Delta^2)$ by exhibiting a collapse to $\M(\partial\Delta^2)$. For any simplex $\sigma$ of $\M(\Delta^2)$, let $\st(\sigma)$ denote the star of $\sigma$ in $\M(\Delta^2)$; that is, the set of simplices of $\M(\Delta^2)$ having $\sigma$ as a face. Any maximal gradient vector field on $\Delta^2$ contains exactly one of the primitive
gradient vector fields $(ab, abc)$, $(bc, abc)$, or $(ac, abc)$, since at most one $2$-simplex can appear in a regular pair. Letting $*$ denote the join $A * B = \{\sigma \cup \tau : \sigma \in A,\ \tau \in B\}$, we have
$$
\begin{aligned}
\st((ab, abc)) &= (ab, abc) * \M(ac, bc), \\
\st((bc, abc)) &= (bc, abc) * \M(ab, ac), \\
\st((ac, abc)) &= (ac, abc) * \M(ab, bc),
\end{aligned}
$$
and $\M(\Delta^2) = \st((ab,abc)) \cup \st((bc,abc)) \cup \st((ac,abc))$.

Each $\M(xy, yz)$ is itself a path of length $3$. Hence each star is a cone over a path of length $3$, with apex the
primitive gradient vector field involving $abc$. By \cite[Lemma 4.2.10]{BarmakThesis}, if $K$ is a simplicial complex and $v$ a vertex of $K$, then the link of $v$ is collapsible if and only if $K$ collapses to $K\backslash \{v\}$. Hence we apply this to each of the three stars above. The apex of each cone lies in that star alone, so the four collapse pairs in each star involve only simplices interior to that star. The three collapses can therefore be performed simultaneously inside $\M(\Delta^2)$
without interference, collapsing each star to its base path:
$$
\begin{aligned}
\st((ab, abc)) &\searrow (c, bc) - (a, ac) - (b, bc) - (c, ac), \\
\st((bc, abc)) &\searrow (a, ab) - (c, ac) - (b, ab) - (a, ac), \\
\st((ac, abc)) &\searrow (b, ab) - (c, bc) - (a, ab) - (b, bc).
\end{aligned}
$$
These three paths together have nine edges on six vertices, and one verifies by direct comparison with the image in part (a) that they are exactly the nine edges of $\M(\partial\Delta^2)$. Hence
$$
\M(\Delta^2) \searrow \M(\partial\Delta^2) \simeq \bigvee^4 S^1.
$$

\item[(c)] In both cases above, every maximum gradient vector field has the same cardinality, so $\M(\Delta^2) = \M_{P}(\Delta^2)$ and $\M(\partial\Delta^2) = \M_{P}(\partial\Delta^2)$.
\end{enumerate}
\end{example}

Starting with $n = 3$, the complexes $\M(\Delta^n)$ and $\M_{P}(\Delta^n)$ cease to be isomorphic, as observed by Chari and Joswig. Nevertheless, in the $n = 3$ case one recovers a homotopy equivalence after passing to the pure subcomplex; see Theorem \ref{thm: h.t. of pure}.

\subsection{Background results}

We collect three results that will be used repeatedly in what follows. The first two, due to Bravo and Camarena, allow one to promote a homology calculation to a homotopy equivalence with a wedge of spheres, provided the space is known to be simply connected. The third, from \cite{scoville2020higher}, provides exactly such a connectivity guarantee for $\M(K)$ in terms of the $1$-skeleton of $K$. Together they form the principles by which we identify the homotopy types of $\M(\Delta^3)$, $\M(\partial\Delta^3)$, $\M_{P}(\Delta^3)$, and $\GM(\Delta^3)$ in Sections \ref{sec:n3} and \ref{sec: GM}.

\begin{thm}\label{thm: wedge of spheres homotopy}\cite[Theorem 1]{BravoCamarena24}
Let $X$ be a simply connected CW complex whose only nonzero reduced homology group is $\tilde{H}_d(X) \cong \ZZ^a$. Then
$$
X \simeq \bigvee_a S^d.
$$
\end{thm}

\begin{thm}\label{thm: double wedge of spheres homotopy}\cite[Theorem 2]{BravoCamarena24}
Let $X$ be a simply connected CW complex such that
$$
\tilde{H}_q(X) =
\begin{cases}
\ZZ^a & \text{for } q = d, \\
\ZZ^b & \text{for } q = d + k, \\
0 & \text{otherwise,}
\end{cases}
$$
where $a, b, d, k$ are positive integers with $d > k$ and $k \in \{1, 5, 6, 13, 62\}$. Then
$$
X \simeq \bigvee_a S^d \vee \bigvee_b S^{d + k}.
$$
\end{thm}

\begin{rem}
In our applications we will only need the case $k = 1$.
\end{rem}

\begin{thm}\label{thm: d-2 connected}\cite[Theorem 2.7]{scoville2020higher}
If $K$ has a vertex of degree $d$ in $K_{(1)}$, then $\M(K)$ is $(d-2)$-connected.
\end{thm}

We will apply Theorem \ref{thm: d-2 connected} most often to $K = \Delta^n$ or $K = \partial\Delta^n$, whose $1$-skeletons are $n$-regular, yielding $(n-2)$-connectivity of $\M(K)$ and in particular simple connectivity as soon as $n \geq 3$.

With this framework in place, we now turn to the structure of $\M(\Delta^n)$ for general $n$. The homotopy type is currently out of reach at this level of generality, but one can already ask a more basic question: Recall that a gradient vector field or discrete Morse matching is called \textbf{optimal} if $\sum m_i$ is minimal where $m_i$ denotes the number of critical simplices of dimension $i$. How many optimal discrete Morse matchings does $\Delta^n$ admit? This is the starting point of the next section.

\section{Towards counting optimal discrete Morse matchings on $\Delta^n$}

Let $f(n)$ denote the number of optimal discrete Morse matchings on $\Delta^n$.  In \cite{CJ-2005}, the authors observe that $f(1) = 2$, $f(2) = 9$, and $f(3) = 256$.  We computed that $f(4) = 380{,}125$ by exhaustive enumeration of all maximum acyclic matchings on the Hasse diagram of $\Delta^4$, carried out on a computer cluster. Our algorithmic method successfully reproduces the known values $f(1) = 2$, $f(2) = 9$, and $f(3) = 256$, and the result for $f(4)$ was independently cross-checked against the bijection of Proposition \ref{prop: top dim bijection} below, which predicts that $\M(\Delta^n)$ and $\M(\partial\Delta^n)$ have the same number of top-dimensional facets; in the case $n = 4$, both sides yield $380{,}125$.

Beyond these values, little is known. Chari and Joswig established the bounds $r(n+1) \leq f(n) \leq (n+1)^{2^{n-1}}$, where $r(n+1)$ is a recursive lower bound \cite[Proposition 5.7, Corollary 5.5]{CJ-2005}. Our computation shows that the upper bound is not tight at $n = 4$, and no closed formula for $f(n)$ is known.

In this section we develop structural results that constrain $f(n)$ for general $n$. We first show that $\M(\Delta^n)$ and $\M(\partial\Delta^n)$ have the same number of top-dimensional facets. We then establish a layer count for optimal matchings and use it to show  that the critical $(n-2)$-faces of an optimal matching on the  $(n-2)$-skeleton of $\Delta^n$ form a spanning tree of $K_{n+1}$. Here we  view the $(n-1)$-faces $F_0, \ldots, F_n$ of $\Delta^n$ as the vertices of  $K_{n+1}$, with an edge between $F_i$ and $F_j$ whenever they share an 
$(n-2)$-face.  These two ingredients combine in our main result of this section, Theorem \ref{thm: morse bijection}, which says that restricting an optimal matching on $\Delta^n$ to $\Delta^n_{(n-2)}$ defines a surjection with uniform fiber size $n+1$, so that
$$
f(n) = (n+1) \cdot \bigl|\text{top-dimensional facets of } \M(\Delta^n_{(n-2)})\bigr|.
$$
Hence the enumeration of optimal matchings on the full $n$-simplex reduces to the same enumeration on its $(n-2)$-skeleton at the cost of a single multiplicative factor of $n+1$.

\subsection{The top-dimensional facet bijection}\label{subsec: top facet bijection}

The following proposition establishes that $f(n)$ equals the number of top-dimensional facets of $\M(\partial\Delta^n)$. The idea is simple: given an optimal matching on $\Delta^n$, remove the $n$-simplex, making the simplex it was matched with critical and thus creating an optimal matching on $\partial \Delta^n$.  Conversely, given an optimal matching on $\partial\Delta^n$, insert the $n$-simplex and match this with the unique critical $(n-1)$-simplex.  See Figure \ref{fig: top facet bijection}. This result is most likely known, but established here for completeness.

\begin{prop}\label{prop: top dim bijection}
For every $n\ge 2$, there is a bijection between the top-dimensional facets of $\M(\Delta^n)$ and those of $\M(\partial\Delta^n)$.
\end{prop}

\begin{proof}
Let $\sigma$ denote the unique $n$-simplex of $\Delta^n$. If $F$ is a top-dimensional facet of $\M(\Delta^n)$, then $F$ is an optimal matching on $\Delta^n$. Since $\Delta^n$ is collapsible, every optimal matching on $\Delta^n$ has exactly one critical simplex which is necessarily a vertex. Hence $\sigma$ is matched to a unique $(n-1)$-face, say $\tau$. Deleting the pair $(\tau,\sigma)$ yields an acyclic matching $\phi(F)$ on $\partial\Delta^n$. Its only critical simplices are the critical vertex of $F$ and $\tau$, so $\phi(F)$ has exactly two critical simplices. Since $\partial\Delta^n\simeq S^{n-1}$, this is minimal, and therefore $\phi(F)$ is a top-dimensional facet of $\M(\partial\Delta^n)$.

Conversely, let $G$ be a top-dimensional facet of $\M(\partial\Delta^n)$. Because $\partial\Delta^n\simeq S^{n-1}$, every optimal matching on $\partial\Delta^n$ has exactly two critical simplices: namely, one critical vertex and one critical $(n-1)$-simplex, say $\tau$. Let $\psi(G)$ be obtained by adjoining the pair $(\tau,\sigma)$. This is again an acyclic matching since any directed cycle would have to involve the new matched edge $\tau\to \sigma$, but after leaving $\sigma$ one enters $\partial\Delta^n$, where no directed path can return to the critical simplex $\tau$. Thus $\psi(G)$ is acyclic, and its only critical simplex is the critical vertex, so it is optimal on $\Delta^n$.

The constructions are plainly inverse to one another, and hence they define a bijection.
\end{proof}

\begin{rem}
In the case $n = 1$, $\Delta^1$ is an edge with two optimal matchings while $\partial\Delta^1$ consists of two isolated vertices, so the proposition fails. This is the reason for the hypothesis $n \geq 2$.
\end{rem}

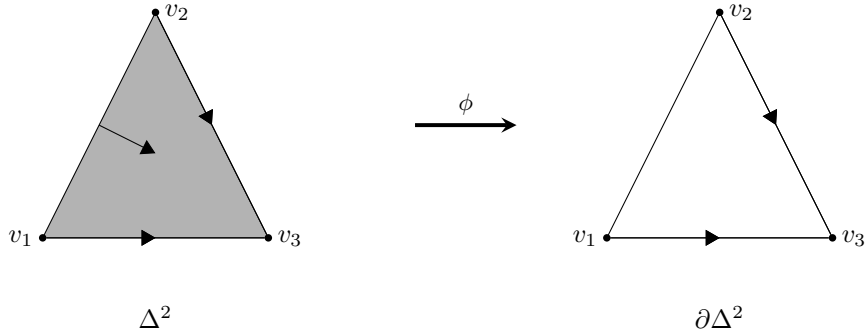
\begin{figure}[h]
\centering
\begin{tikzpicture}[scale=1.5,
    decoration={markings,mark=at position .5 with {\Huge{\arrow{triangle 60}}}},]

\begin{scope}[shift={(-2.5,0)}]
\filldraw[fill=black!30, draw=black] (-1,0)--(1,0)--(0,2)--cycle;
\node[inner sep=1pt, circle, fill=black](v1) at (-1,0) {};
\node[inner sep=1pt, circle, fill=black](v2) at (1,0) {};
\node[inner sep=1pt, circle, fill=black](v3) at (0,2) {};
\draw[->-]  (v1)--(v2) node[midway, below] {} ;
\draw[->-]  (v3)--(v2) node[midway, left] {} ;
\draw[-triangle 60]  (-.5,1)--(0,.75) node[above] {} ;
\node[anchor = east]  at (v1) {{$v_1$}};
\node[anchor = west]  at (v2) {{$v_3$}};
\node[anchor = west]  at (v3) {{$v_2$}};
\node at (0,-0.7) {$\Delta^2$};
\end{scope}

\draw[->, >=stealth, line width=1.5pt] (-.2,1) -- (.7,1)
    node[midway, above] {\small $\phi$};

\begin{scope}[shift={(2.5,0)}]
\draw[black] (-1,0)--(1,0)--(0,2)--cycle;
\node[inner sep=1pt, circle, fill=black](w1) at (-1,0) {};
\node[inner sep=1pt, circle, fill=black](w2) at (1,0) {};
\node[inner sep=1pt, circle, fill=black](w3) at (0,2) {};
\draw[->-]  (w1)--(w2) node[midway, below] {} ;
\draw[->-]  (w3)--(w2) node[midway, left] {} ;
\node[anchor = east]  at (w1) {{$v_1$}};
\node[anchor = west]  at (w2) {{$v_3$}};
\node[anchor = west]  at (w3) {{$v_2$}};
\node at (0,-0.7) {$\partial\Delta^2$};
\end{scope}

\end{tikzpicture}
\caption{The top-facet bijection $\phi$ of
Proposition \ref{prop: top dim bijection} for $n = 2$. Left: An optimal matching on $\Delta^2$ with matched pairs $(v_1, v_1v_3)$, $(v_2, v_2v_3)$, and $(v_1v_2, v_1v_2v_3)$; the unique critical simplex is $v_3$. Right: The image $\phi(F)$ on $\partial\Delta^2$, obtained by deleting the pair $(v_1v_2, v_1v_2v_3)$ involving the $2$-simplex. The surviving pairs $(v_1, v_1v_3)$ and $(v_2, v_2v_3)$ form an optimal matching on $\partial\Delta^2$ with two critical simplices: $v_3$ and $v_1v_2$.}
\label{fig: top facet bijection}
\end{figure}

\subsection{Layer count and spanning tree structure}\label{subsec: layer count}

We now turn from enumeration on $\Delta^n$ itself to enumeration on its  $(n-2)$-skeleton. We first record the result that optimal matchings on $\Delta^n$ have a rigid layer structure in the sense that the number of pairs matching a $k$-face with a $(k+1)$-face is completely determined by $n$ and $k$.

\begin{lem}\label{lem: counting pairs}
Let $V$ be an optimal discrete Morse matching on $\Delta^n$, with critical vertex $v$. For each $k=0,1,\dots,n-1$, the number of pairs $(\sigma^{(k)},\tau^{(k+1)})\in V$ is exactly $\binom{n}{k+1}$.
\end{lem}

\begin{proof}
Since $\Delta^n$ is collapsible and $V$ is optimal, the only critical simplex is $v$. Hence every simplex other than $v$ is matched.

We proceed by induction on $k$. For $k=0$, every vertex except $v$ must be matched upward to an edge, so the number of $(0,1)$-pairs is $n=\binom{n}{1}$.

Now let $1\leq k\leq n-1$, and assume that the number of $(k-1,k)$-pairs is $\binom{n}{k}$. Each such pair uses one $k$-face as its upper element, so exactly $\binom{n}{k}$ many $k$-faces are matched downward. Since $\Delta^n$ has $\binom{n+1}{k+1}$ faces of dimension $k$, the remaining
$$
\binom{n+1}{k+1}-\binom{n}{k}=\binom{n}{k+1}
$$
$k$-faces cannot be critical and therefore must be matched upward to $(k+1)$-faces. Thus the number of $(k,k+1)$-pairs is $\binom{n}{k+1}$.
\end{proof}

Lemma \ref{lem: counting pairs} shows that optimal matchings on $\Delta^n$ are tightly constrained. We now use this rigidity to relate optimal matchings on $\Delta^n$ to those on its $(n-2)$-skeleton $\Delta^n_{(n-2)}$. Recall from \cite[Theorem 5.2]{zax2012morse} that $\Delta^n_{(n-2)} \simeq \bigvee^n S^{n-2}$, so any optimal matching on $\Delta^n_{(n-2)}$ has exactly $n$ critical $(n-2)$-faces in addition to the critical vertex.

As mentioned above, the $(n-2)$-faces of $\Delta^n$ admit a natural identification with the edges  of the complete graph $K_{n+1}$. Let $\mathcal{F} = \{F_0, \ldots, F_n\}$ denote the $(n-1)$-dimensional simplices of $\Delta^n$. Then each $(n-2)$-face of $\Delta^n$ lies in exactly two elements of $\mathcal{F}$, and viewing $\mathcal{F}$ as the vertex set of a graph with an edge between $F_i$ and $F_j$ whenever they share an $(n-2)$-face, we obtain $K_{n+1}$. Under this identification we may ask what the critical $(n-2)$-faces of an optimal matching on $\Delta^n_{(n-2)}$ look like as a subgraph of $K_{n+1}$. The answer is the content of the next lemma.

\begin{lem}\label{lem: spanning tree}
Let $n \ge 3$, let $O$ be an optimal discrete Morse matching on $\Delta^n_{(n-2)}$, and let $U=\{u_1,\dots,u_n\}$ be its critical $(n-2)$-faces. Under the natural identification of the $(n-2)$-faces of $\Delta^n$ with the edges of $K_{n+1}$, the set $U$ forms a spanning tree of $K_{n+1}$.
\end{lem}

\begin{proof}
Let $\mathcal{F}=\{F_0,\dots,F_n\}$ be the set of $(n-1)$-faces of $\Delta^n$, and let $G=(\mathcal{F},U)$ be the graph whose vertices are the $F_i$ and whose edges are the critical $(n-2)$-faces. Since $\Delta^n_{(n-2)} \simeq \bigvee^n S^{n-2}$, every optimal matching on $\Delta^n_{(n-2)}$ has exactly one critical vertex and exactly $n$ critical $(n-2)$-faces. Thus $G$ has $n+1$ vertices and $n$ edges, so it suffices to show that $G$ is connected.

Suppose by way of contradiction that $G$ is disconnected. Let $C \subsetneq \mathcal{F}$ be the vertex set of a connected component. Define
$$
z_C = \sum_{F_i \in C} (-1)^i \partial F_i
\in C_{n-2}(\Delta^n_{(n-2)};\mathbb{Z}).
$$
Since each $\partial F_i$ is an $(n-2)$-cycle, so is $z_C$. We claim that every critical $(n-2)$-face has coefficient 0 in $z_C$. Let $u \in U$ be a critical $(n-2)$-face. Then $u$ lies in exactly two facets, say $F_i$ and $F_j$, and it appears in $\partial F_i$ and $\partial F_j$ with opposite signs. Since $u$ is an edge of $G$ and $C$ is a connected component, either both $F_i, F_j$ lie in $C$ or neither does. In the first case the two contributions cancel in $z_C$, and in the second case $u$ does not appear at all. Hence every critical $(n-2)$-face has coefficient 0 in $z_C$.

Now let $Z_{n-2}$ denote the group of $(n-2)$-cycles, and let
$$
\pi \colon Z_{n-2} \longrightarrow \mathbb{Z}^U
$$
be the projection onto the coordinates indexed by the critical $(n-2)$-faces. Since $\Delta^n_{(n-2)}$ has no $(n-1)$-cells, $B_{n-2}(\Delta^n_{(n-2)}) = 0$, so
$$
Z_{n-2} \cong H_{n-2}(\Delta^n_{(n-2)}) \cong \mathbb{Z}^n.
$$
We now claim that $\pi$ is injective. Let $B$ denote the set of matched $(n-2)$-faces, and let $A$ denote the set of $(n-3)$-faces that are paired with elements of $B$ under $O$. The submatrix $\partial_{B,A}$ of the boundary operator $\partial_{n-2}$ with columns indexed by $B$ and rows indexed by $A$ is invertible over $\mathbb{Z}$. Indeed, acyclicity of $O$ allows the matched pairs to be ordered so that $\partial_{B,A}$ is upper triangular with $\pm 1$ on the diagonal. For any cycle $z = (z_B, z_U) \in Z_{n-2}$, the condition $\partial z = 0$ restricted to the rows in $A$ gives $\partial_{B,A}(z_B) + \partial_{U,A}(z_U) = 0$, so $z_B =  (\partial_{B,A})^{-1} \partial_{U,A}(z_U)$, and $z$ is uniquely determined by $z_U = \pi(z)$. Hence $\pi$ is injective.

Since all critical coordinates of $z_C$ vanish, we have that $\pi(z_C) = 0$, and injectivity of $\pi$ gives $z_C = 0$. On the other hand, the cycles $\partial F_0, \ldots, \partial F_n$ satisfy exactly one linear relation, namely,
$$
\sum_{i=0}^n (-1)^i \partial F_i = 0,
$$
coming from the boundary of the $n$-simplex. Since $C$ is a nonempty proper subset of $\mathcal{F}$, the partial sum $z_C$ cannot vanish. This contradiction shows that $G$ is connected.

We conclude that since $G$ is a connected graph on $n+1$ vertices with $n$ edges, it is a spanning tree.
\end{proof}

\begin{example}
We illustrate Lemma \ref{lem: spanning tree} for $n = 3$. Here $\Delta^3_{(1)}$ is the complete graph $K_4$ on vertices $v_0, v_1, v_2, v_3$. The four facets of $\Delta^3$ are the triangles $F_i = v_0 \cdots \hat{v}_i \cdots v_3$, the six $(n-2)$-faces are the edges $e_{ij} = v_i v_j$, and the four $(n-3)$-faces are the vertices $v_i$.

Consider the optimal matching on $\Delta^3_{(1)}$ given by
$$
(v_1, e_{01}), \qquad (v_2, e_{02}), \qquad (v_3, e_{03}).
$$
The unique critical vertex is $v_0$, and the critical edges are $U = \{e_{12}, e_{13}, e_{23}\}$. Under the identification of edges of $\Delta^3$ with edges of the graph $G = (\mathcal{F}, U)$, each $e_{ij}$ joins the two facets containing it:
$$
e_{12} \leftrightarrow F_0 F_3, \qquad
e_{13} \leftrightarrow F_0 F_2, \qquad
e_{23} \leftrightarrow F_0 F_1.
$$
These three edges form a star centered at $F_0$, which is a spanning tree of $K_4$. See Figure \ref{fig: spanning tree}.

\begin{figure}[h]
\centering
\begin{tikzpicture}[scale=1.4,
    ->-/.style={decoration={markings, mark=at position 0.5 with 
    {\arrow{triangle 60}}}, postaction={decorate}}]

\begin{scope}[shift={(-2.5,0)}]

\coordinate (v0) at (0,2.4);
\coordinate (v1) at (-1.6,0);
\coordinate (v2) at (1.6,0);
\coordinate (v3) at (0,0.9);

\draw[black!70] (v0)--(v2);
\draw[black!70] (v1)--(v2);
\draw[black!70] (v1)--(v3);
\draw[black!70] (v2)--(v3);
\draw[black!70] (v0)--(v1);
\draw[black!70] (v0)--(v3);

\draw[->-] (v1)--(v0);
\draw[->-] (v2)--(v0);
\draw[->-] (v3)--(v0);

\draw[black] (v1)--(v2);
\draw[black] (v1)--(v3);
\draw[black] (v2)--(v3);

\node[inner sep=2pt, circle, fill=black] at (v0) {};
\node[inner sep=2pt, circle, fill=black] at (v1) {};
\node[inner sep=2pt, circle, fill=black] at (v2) {};
\node[inner sep=2pt, circle, fill=black] at (v3) {};

\node[anchor=south] at (v0) {$v_0$};
\node[anchor=east] at (v1) {$v_1$};
\node[anchor=west] at (v2) {$v_2$};
\node[anchor=north] at ($(v3)+(0,-0.15)$) {$v_3$};

\node[anchor=north, black] at ($(v1)!0.5!(v2)+(0,-0.1)$) {\small $e_{12}$};
\node[anchor=east, black] at ($(v1)!0.5!(v3)+(0.6,0)$) {\small $e_{13}$};
\node[anchor=west, black] at ($(v2)!0.5!(v3)+(-.6,0)$) {\small $e_{23}$};

\node at (0,-1) {$\Delta^3_{(1)}$};
\end{scope}

\draw[->, >=stealth, line width=1.2pt] (-0.3,1.0) -- (0.3,1.0);

\begin{scope}[shift={(2.5,0)}]

\coordinate (F0) at (0,2.4);
\coordinate (F1) at (-1.5,0.3);
\coordinate (F2) at (1.5,0.3);
\coordinate (F3) at (0,1.0);

\draw[black!50] (F1)--(F2);
\draw[black!50] (F1)--(F3);
\draw[black!50] (F2)--(F3);

\draw[black, line width=2pt] (F0)--(F1)
    node[midway, anchor=east] {\small $e_{23}$};
\draw[black, line width=2pt] (F0)--(F2)
    node[midway, anchor=west] {\small $e_{13}$};
\draw[black, line width=2pt] (F0)--(F3)
    node[midway, anchor=east] {\small $e_{12}$};

\node[anchor=north, black] at ($(F1)!0.5!(F2)+(0,-0.05)$) 
    {\small $e_{03}$};
\node[anchor=east, black] at ($(F1)!0.5!(F3)+(0.1,.1)$) 
    {\small $e_{02}$};
\node[anchor=west, black] at ($(F2)!0.5!(F3)+(0,0)$) 
    {\small $e_{01}$};

\node[inner sep=2pt, circle, fill=black] at (F0) {};
\node[inner sep=2pt, circle, fill=black] at (F1) {};
\node[inner sep=2pt, circle, fill=black] at (F2) {};
\node[inner sep=2pt, circle, fill=black] at (F3) {};

\node[anchor=south] at (F0) {$F_0$};
\node[anchor=east] at (F1) {$F_1$};
\node[anchor=west] at (F2) {$F_2$};
\node[anchor=north] at ($(F3)+(0,-0.15)$) {$F_3$};

\node at (0,-1) {$K_4$};
\end{scope}

\end{tikzpicture}
\caption{Illustration of Lemma \ref{lem: spanning tree} for $n = 3$. Left: An optimal matching on $\Delta^3_{(1)}$ with matched pairs $(v_1, e_{01})$, $(v_2, e_{02})$, $(v_3, e_{03})$ and critical vertex $v_0$. Right: The complete graph $K_4$ on the vertex set $\{F_0, F_1, F_2, F_3\}$, where each edge is labeled by the $(n-2)$ face of $\Delta^3$ it represents. The thickened edges are those corresponding to the critical edges of the matching, forming a spanning tree centered at $F_0$.}
\label{fig: spanning tree}
\end{figure}
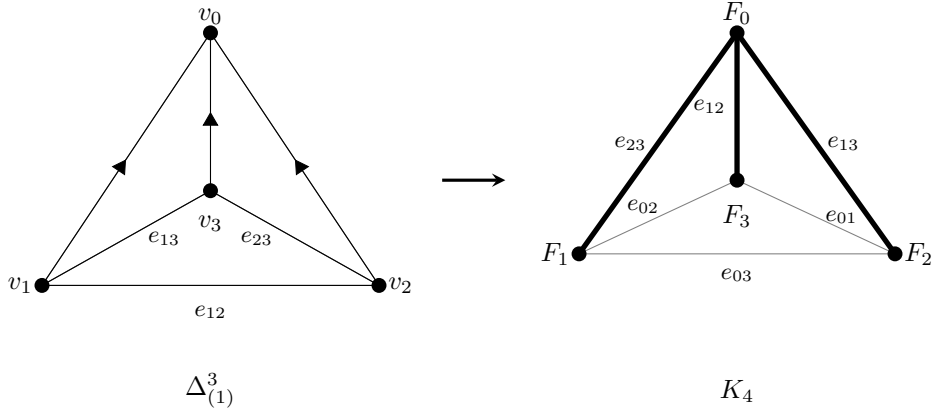

To see the idea of the proof, suppose we tried to split $\mathcal{F}$ into $C = \{F_0\}$ and $\mathcal{F} \setminus C = \{F_1, F_2, F_3\}$. Then
$$
z_C = (-1)^0 \partial F_0 = \partial(v_1 v_2 v_3) = e_{23} - e_{13} + e_{12}.
$$
The critical edges are $e_{12}, e_{13}, e_{23}$, and all three appear in $z_C$ with nonzero coefficients which does not satisfy the hypothesis of the lemma. 

If instead we tried the partition $C = \{F_0, F_3\}$ and $\mathcal{F} \setminus C = \{F_1, F_2\}$, then
$$
\begin{aligned}
z_C &= \partial F_0 - \partial F_3 \\
&= (e_{23} - e_{13} + e_{12}) - (e_{12} - e_{02} + e_{01})\\
&= e_{23} - e_{13} + e_{02} - e_{01}.
\end{aligned}
$$
The critical edges $e_{12}, e_{13}, e_{23}$ have coefficients $0, -1, 1$ respectively, so again $z_C$ has nonzero critical support. This too is consistent: the edge $e_{13} \leftrightarrow F_0 F_2$ crosses the cut between $C$ and its complement, confirming that $C$ is not a union of connected components of $G$.
\end{example}

\subsection{Optimal matching bijection}\label{subsec: restriction}

We have shown that the critical $(n-2)$-faces of an optimal matching on $\Delta^n_{(n-2)}$ form a spanning tree of $K_{n+1}$. We now turn to the converse direction: every optimal matching on $\Delta^n$ restricts to an optimal matching on $\Delta^n_{(n-2)}$, and we count how many optimal matchings on $\Delta^n$ project to a given optimal matching on the $(n-2)$-skeleton.

\begin{prop}\label{prop: restricting delta n}
Let $n\geq 3$. Restricting an optimal discrete Morse matching on $\Delta^n$ to the subcomplex $\Delta^n_{(n-2)}$ yields an optimal discrete Morse matching on $\Delta^n_{(n-2)}$ with exactly one critical vertex and $n$ critical $(n-2)$-faces.
\end{prop}

\begin{proof}
By Lemma \ref{lem: counting pairs}, an optimal matching on $\Delta^n$ contains exactly $\binom{n}{n-1} = n$ pairs of the form $(\sigma^{(n-2)}, \tau^{(n-1)})$ and exactly $\binom{n}{n} = 1$ pair of the form $(\tau^{(n-1)}, \sigma^{(n)})$. Removing these $n+1$ pairs from the matching yields a matching on $\Delta^n_{(n-2)}$ in which the $n$ $(n-2)$-faces that were previously matched upward become critical, and the unique critical vertex is preserved. Since $\Delta^n_{(n-2)} \simeq \bigvee^n S^{n-2}$ requires at least $n$ critical $(n-2)$-simplices and one critical vertex in any discrete Morse matching, the restricted matching achieves the Morse lower bound and is therefore optimal.
\end{proof}

Let $\rho$ denote the restriction map of Proposition \ref{prop: restricting delta n} from optimal matchings on $\Delta^n$ to optimal matchings on $\Delta^n_{(n-2)}$. We now prove the main result of this section that $\rho$ is surjective and each fiber has size exactly $n+1$. 

\begin{thm}\label{thm: morse bijection}
The map $\rho$ is surjective with uniform fiber size $n+1$, so
$$
f(n) = (n+1) \cdot \bigl|\text{top-dimensional facets of } \M(\Delta^n_{(n-2)})\bigr|.
$$
\end{thm}

\begin{proof}
Let $O$ be an optimal matching on $\Delta^n_{(n-2)}$ with critical $(n-2)$-faces $U = \{u_1, \ldots, u_n\}$. By Lemma \ref{lem: spanning tree}, $U$ is the edge set of a spanning tree $T$ of $K_{n+1} = (\{F_0, \ldots, F_n\}, \binom{\mathcal{F}}{2})$.

We first extend $O$ to an optimal matching on $\Delta^n$. Choose a facet  $F_r \in \mathcal{F}$ to be the root; the choice of root is where the  factor $n+1$ in the formula originates. Orient every edge of $T$ toward $F_r$. This orientation determines $n$ new matched pairs on $\Delta^n$ by specifying that for each edge $u = F_iF_j$ of $T$ oriented from $F_i$ to $F_j$, match the $(n-2)$-face $u$ with the $(n-1)$-face $F_i$. This leaves the root $F_r$ as the only unmatched facet. Finally, match $F_r$ with the $n$-simplex $\sigma$. Let $P$ denote the resulting matching on $\Delta^n$, which combines $O$ with these $n+1$ new pairs.

We first show that $P$ is acyclic. Suppose by contradiction that the modified Hasse diagram of $P$ contains a directed cycle $C$. Since $O$ is acyclic, $C$ must use at least one newly added upward edge. Now $\sigma$ cannot lie on $C$ since the only upward edge into $\sigma$ is $F_r \to \sigma$, and all edges from $\sigma$ point downward to $(n-1)$-faces $F_i \neq F_r$. Hence any path through $\sigma$ must leave along a downward edge and cannot return so that $C$ lies entirely in dimensions $\leq n - 1$ and uses at least one upward edge $u \to F_i$ with $u \in U$ and $(u, F_i)$ a newly added pair.

For each facet $F \in \mathcal{F}$, let $d(F)$ denote the distance from $F$ to the root $F_r$ in the tree $T$. We claim that every time $C$ traverses a newly added upward edge $u \to F_i$, the next facet visited by $C$ (if any) must have a strictly larger distance from the root. Indeed, suppose $u = F_i F_j$ is matched with $F_i$. By our construction, $u$ is oriented toward the root, meaning $F_j$ is the parent of $F_i$ and $d(F_i) = d(F_j) + 1$. After entering $F_i$ via $u \to F_i$, the cycle $C$ must exit $F_i$ via a downward edge to some $(n-2)$-face $v \in \partial F_i$. We consider two cases: either $v \in U$ or $v \notin U$.

Suppose $v = u'$ for some $u' \in U$ distinct from $u$. Since $C$ must continue, it must travel up a matched edge $u' \to F_k$. By our matching rule, $u'$ is oriented from $F_k$ to its parent. Because $u'$ connects $F_i$ and $F_k$, and $F_j$ is already the unique parent of $F_i$, it must be that $F_i$ is the parent of $F_k$. Thus $d(F_k) = d(F_i) + 1$, meaning the distance from the root has strictly increased.

Now suppose $v \notin U$. Then $v$ is matched in $O$ with an $(n-3)$-face, forcing $C$ into the lower skeleton. To return to the facet layer, $C$ would have to traverse an upward edge from the lower skeleton into some $u'' \in U$. However, the faces in $U$ are critical in $O$ and therefore unmatched in the lower skeleton. Consequently, there are no upward edges entering $U$ from below. Any path entering the lower skeleton is permanently trapped and cannot participate in a cycle involving the facet layer.

In either case, the distance $d$ strictly increases with every new facet visited. Since $T$ is finite, $d$ is bounded above by the maximum depth of the tree, so no such sequence can close into a cycle. Hence $P$ is acyclic.

To see that $P$ is optimal, observe that the matching $P$ has exactly one critical simplex since the $n+1$ new pairs consume exactly the $F_r$, the remaining $n$ facets, all $n$ previously critical $(n-2)$-faces, and $\sigma$. Since $\Delta^n$ is collapsible, a matching with a single critical vertex is optimal so $P \in \rho^{-1}(O)$.

Now to determine the size of each fiber, we first observe that the construction of $P$ depends only on the choice of root $F_r$, and different  roots yield different matchings. This gives $n + 1$ extensions of $O$ to optimal matchings on $\Delta^n$.

Conversely, let $P' \in \rho^{-1}(O)$. Then $\sigma$ is matched in $P'$ with some $(n-1)$-face $F_r$, which we designate the root. Each remaining facet $F_i \neq F_r$ must be matched in $P'$ with some $(n-2)$-face, and since $\Delta^n_{(n-2)}$ is already optimally matched by $O$, the face matched with $F_i$ must be one of the critical faces $u \in U$, specifically one incident to $F_i$ as an edge of $T$. Matching each non-root facet with exactly one of its incident tree edges amounts to orienting every edge of $T$ toward exactly one endpoint so that every non-root vertex receives exactly one incoming arrow. Because $T$ is a tree, the only such orientation directs every edge toward $F_r$. Hence $P'$ arises from our construction for the root $F_r$, and the fiber $\rho^{-1}(O)$ has exactly $n + 1$ elements.

Since this holds for every $O$, the map $\rho$ is surjective with uniform fiber size $n + 1$, and the formula for $f(n)$ follows.
\end{proof}

\begin{rem}
In the $n = 3$ case, we have that $\M(\Delta^3_{(1)})$ has $\frac{f(3)}{4} = 64$ top-dimensional facets. This can be verified directly from the $K_4$ spanning tree count as there are $4^{4-2} = 16$ spanning trees of $K_4$ by Cayley's formula, and each extends to $4$ optimal matchings on $\Delta^3_{(1)}$ via the choice of critical vertex.
\end{rem}

\section{Inclusions and the cone case}

The simplicial complex $\Delta^n$ admits a natural family of subcomplexes in its skeleta and its boundary.  An even more fundamental decomposition is given by $\Delta^n = v*\Delta^{n-1}$ as a cone over $\Delta^{n-1}$.

In this section we develop two tools that exploit this structure. The first is the general result that any subcomplex inclusion $K \subseteq L$ lifts to a simplicial inclusion $\M(K) \hookrightarrow \M(L)$, yielding a long exact sequence on homology. The second tool is specific to cones: when $L = CK$, the lifted inclusion is null-homotopic, and the long exact sequence consequently splits. Both apply to $\Delta^n$, the second via the cone decomposition $\Delta^n = v * \Delta^{n-1}$.

\subsection{Lifting an inclusion}\label{subsec: lifting inclusion}

Recall that a subcomplex $T \subseteq K$ is called \textbf{full} if every simplex of $K$, all of whose vertices lie in $T$, is itself contained in $T$.

\begin{prop}\label{prop: inclusion lifts to Morse}
Let $i \colon K \hookrightarrow L$ be the inclusion of a subcomplex of a simplicial complex $L$. Then the induced map $i_* \colon \M(K) \to \M(L)$ is an injective simplicial map, and its image is a full subcomplex of $\M(L)$.
\end{prop}

\begin{proof}
A vertex of $\M(K)$ is a primitive gradient vector field $(\sigma, \tau)$ with $\sigma \subseteq \tau$ and $\dim \tau = \dim \sigma + 1$. Define $i_*$ on vertices by $i_*(\sigma, \tau) = (\sigma, \tau)$, viewed as a vertex of $\M(L)$. This is well-defined since $\sigma, \tau \in K \subseteq L$, and clearly injective.

Let $V = \{(\sigma_0, \tau_0), \ldots, (\sigma_m, \tau_m)\}$ be an $m$-simplex of $\M(K)$. Since $V$ is a discrete vector field on $K$, its pairs are mutually disjoint, so $V$ is a valid discrete vector field on $L$. To see that $V$ is acyclic on $L$, observe that any non-trivial closed $V$-path
$$
\alpha_0^{(p)}, \beta_0^{(p+1)}, \alpha_1^{(p)}, \beta_1^{(p+1)}, \ldots, \alpha_{k-1}^{(p)}, \beta_{k-1}^{(p+1)}, \alpha_k^{(p)} = \alpha_0^{(p)}
$$
on $L$ has every $(\alpha_i, \beta_i) \in V$, so all $\alpha_i, \beta_i$ lie in $K$. Furthermore, since $\alpha_{i+1}$ is a face of $\beta_i$ and $K$ is closed under taking faces, every simplex in the path lies in $K$. Such a closed $V$-path on $K$ contradicts acyclicity of $V$ on $K$. Hence $V$ is acyclic on $L$, so $i_*$ sends simplices to simplices.

To show that the image $i_*(\M(K))$ is a full subcomplex, let $T$ be a simplex of $\M(L)$ all of whose vertices are pairs $(\sigma, \tau)$ with $\sigma, \tau \in K$. Because $T$ is a simplex of $\M(L)$, it is a valid matching and its pairs are mutually disjoint. Furthermore, since the Hasse diagram of $T$ on $L$ contains no directed cycles, it contains no directed cycles when restricted to $K$. Thus $T$ is a valid acyclic matching on $K$, and hence it is the image of a simplex in $\M(K)$.
\end{proof}

Proposition \ref{prop: inclusion lifts to Morse} immediately yields a long exact sequence in homology for the pair $(\M(L), \M(K))$ whenever $K \subseteq L$. Applying this to the skeletal filtration
$$
\Delta^n_{(1)} \subseteq \Delta^n_{(2)} \subseteq \cdots \subseteq \Delta^n_{(n-1)} = \partial\Delta^n \subseteq \Delta^n
$$
gives a chain of long exact sequences relating the homology of $\M(\Delta^n)$ to that of its subcomplexes. We illustrate with the pair $(\M(\Delta^3), \M(\partial\Delta^3))$.

\begin{example}\label{ex: LES n=3}
Consider the long exact homology sequence of the pair $(\M(\Delta^3), \M(\partial\Delta^3))$:
$$
\cdots \to H_k(\M(\partial\Delta^3)) \to H_k(\M(\Delta^3)) \to H_k(\M(\Delta^3), \M(\partial\Delta^3)) \to H_{k-1}(\M(\partial\Delta^3)) \to \cdots .
$$
The simplices of the relative complex $\M(\Delta^3)$ relative to $\M(\partial\Delta^3)$ correspond precisely to those discrete vector fields on $\Delta^3$ that involve the unique $3$-simplex. Theorems \ref{thm: h.t. of delta3} and \ref{thm: h.t. of partial delta3} (see below) give $H_k(\M(\partial\Delta^3)) = 0$ for $k \notin \{3, 4\}$ with $H_3 \cong \ZZ^{21}$ and $H_4 \cong \ZZ^{24}$, and $H_k(\M(\Delta^3)) = 0$ for $k \neq 4$ with $H_4 \cong \ZZ^{99}$. A direct computation yields
$$
H_4(\M(\Delta^3), \M(\partial\Delta^3)) \cong \ZZ^{96} \quad\text{and}\quad H_5(\M(\Delta^3), \M(\partial\Delta^3)) = 0,
$$
so we obtain the long exact sequence
$$
0 \to \ZZ^{24} \to \ZZ^{99} \to \ZZ^{96} \to \ZZ^{21} \to 0.
$$
In particular, the induced map $H_4(\M(\partial\Delta^3)) \to H_4(\M(\Delta^3))$ is injective.
\end{example}

The long exact sequence of the skeletal filtration gives constraints on $H_*(\M(\Delta^n))$ but does not, by itself, determine the homotopy type. The next subsection shows that for the cone decomposition $\Delta^n = v * \Delta^{n-1}$, a much stronger statement holds: the inclusion $\M(\Delta^{n-1}) \hookrightarrow \M(\Delta^n)$ is null-homotopic so that the long exact sequence splits.

\subsection{A null-homotopy theorem}\label{subsec: null-homotopy}

The simplicial complex $\Delta^n$ admits a recursive decomposition $\Delta^n = v * \Delta^{n-1}$, expressing the $n$-simplex as a cone over its $(n-1)$-dimensional facet. By Proposition \ref{prop: inclusion lifts to Morse}, the inclusion $\Delta^{n-1} \subseteq \Delta^n$ lifts to a simplicial inclusion $\M(\Delta^{n-1}) \hookrightarrow \M(\Delta^n)$. The main result of this section shows that this inclusion is null-homotopic. More generally, the analogous inclusion $\M(K) \hookrightarrow \M(CK)$ for any cone is null-homotopic.

This is somewhat surprising. The inclusion of the base $K \subseteq CK$ is itself null-homotopic for the trivial reason that $CK$ is contractible, but $\M(CK)$ is generally far from contractible, as verified in even the smallest interesting case $C\Delta^1 = \Delta^2$ (see Example \ref{ex: delta2}), where $\M(\Delta^2) \simeq \bigvee^4 S^1$. So the null-homotopy of $i_*$ does not come from the contractibility of the ambient complex. Rather, it comes from a single, explicit choice of a Hasse edge that is compatible with every matching on $K$, yielding a one-step contiguity.

\begin{thm}\label{thm: null-homotopy}
Let $K$ be a finite, nonempty simplicial complex and $CK = v * K$ its cone with apex $v$. Then the induced inclusion $i_* \colon \M(K) \hookrightarrow \M(CK)$ is null-homotopic.
\end{thm}

\begin{proof}
We will exhibit a single explicit Hasse edge that yields a contiguity between $i_*$ and a constant map.

Recall that two simplicial maps $f, g \colon X \to Y$ are \textbf{contiguous} if for every simplex $\sigma$ of $X$, the union $f(\sigma) \cup g(\sigma)$ is a simplex of $Y$ \cite[Section 3.5]{Spanier}.

Fix a vertex $w_0 \in K$ and consider the matched pair $e_0 = (v, vw_0)$ in the Hasse diagram of $CK$. This $e_0$ is a vertex of $\M(CK)$.  Define the constant map $c \colon \M(K) \to \M(CK)$ by $c(\mu) = e_0$ for every simplex $\mu$ of $\M(K)$.

To show $i_*$ and $c$ are contiguous, we show that for every simplex $\mu \in \M(K)$, the union $i_*(\mu) \cup c(\mu) = \mu \cup \{e_0\}$ is a simplex of $\M(CK)$; that is, it is an acyclic matching on the Hasse diagram of $CK$.

We first show that $\mu \cup \{e_0\}$ is a valid matching. The simplices $v$ and $vw_0$ both contain $v$ and so neither lies in $K$. Since $\mu$ uses only faces of $K$, no simplex appears in both $\mu$ and $\{e_0\}$, so the matched pairs of $\mu \cup \{e_0\}$ are mutually disjoint. 

To see that $\mu \cup \{e_0\}$ is acyclic, suppose by way of contradiction that the modified Hasse diagram of $CK$ with respect to $\mu \cup \{e_0\}$ contains a directed cycle $\gamma$. Recall that unmatched edges are directed downward and matched edges upward. Since $\mu$ is acyclic on $K$, $\gamma$ cannot lie entirely on the sub-Hasse diagram of $K$. Thus $\gamma$ must use the new matched edge $e_0$, traversed upward from $v$ to $vw_0$. Now in the modified Hasse diagram, the edge to $v$ is matched and so directed into $vw_0$. Every coface $vw_0w$ of $vw_0$ lies outside $K$ and outside $\{e_0\}$, so the pair $(vw_0, vw_0 w)$ is unmatched and the edge is directed $vw_0 w \to vw_0$ into $vw_0$. The only outgoing edge from $vw_0$ is therefore the unmatched downward edge to $w_0 \in K$, so $\gamma$ must continue from $vw_0$ into $K$. But once a directed path enters $K$ it cannot leave. Indeed, the matched edges of $\mu$ remain in $K$ so that every Hasse edge from a simplex $\sigma \in K$ to a simplex outside $K$ has the form $\sigma \subseteq \sigma \cup \{v\}$, which is unmatched and therefore directed downward from $\sigma \cup \{v\}$ to $\sigma$. Hence $\gamma$ cannot return to $v$, contradicting the assumption that $\gamma$ is a cycle. Thus $\mu \cup \{e_0\}$ is acyclic and is a simplex of $\M(CK)$.

We conclude that $i_*$ and $c$ are contiguous, and hence $i_*$ is homotopic to a constant map.
\end{proof}

The null-homotopy of $i_*$ has immediate homological and homotopical consequences.

\begin{cor}\label{cor:null-homotopy-consequences}
Let $K$ be a finite, nonempty simplicial complex and $CK = v * K$ its cone. Then
\begin{enumerate}
\item The induced map $i_* \colon \tilde{H}_*(\M(K)) \to \tilde{H}_*(\M(CK))$
is the zero map.
\item The cofiber of $i$ splits as
$$
\M(CK) / \M(K) \simeq \M(CK) \vee \Sigma \M(K),
$$
where $\Sigma$ denotes the (reduced) suspension. Consequently,
$$
\tilde{H}_k(\M(CK), \M(K)) \cong \tilde{H}_k(\M(CK)) \oplus \tilde{H}_{k-1}(\M(K)),
$$
which holds over any coefficient group.
\item The long exact sequence of the pair $(\M(CK), \M(K))$ in reduced homology splits into short exact sequences
$$
0 \to \tilde{H}_k(\M(CK)) \to \tilde{H}_k(\M(CK), \M(K)) \to \tilde{H}_{k-1}(\M(K)) \to 0
$$

for all $k$.
\end{enumerate}
\end{cor}

\begin{proof}
A null-homotopic map induces the zero map on reduced homology, giving (1). For (2), because $i$ is a cofibration, the cofiber $\M(CK)/\M(K)$ is homotopy equivalent to the mapping cone $C_i$. Since $i$ is null-homotopic, $C_i$ is homotopy equivalent to the mapping cone of the constant map, which is exactly $\M(CK) \vee \Sigma\M(K)$. The homology splitting follows from $\tilde{H}_k(\M(CK)/\M(K)) \cong \tilde{H}_k(\M(CK)) \oplus \tilde{H}_{k-1}(\M(K))$, since reduced homology distributes over wedge sums. Statement (3) then follows from (1) applied to the long exact sequence of the pair.
\end{proof}

The application to $\Delta^n$ is then immediate.

\begin{cor}\label{cor:simplex-induction}
For all $n \geq 2$, the inclusion $\M(\Delta^{n-1}) \hookrightarrow \M(\Delta^n)$ is null-homotopic, the long exact sequence of the pair
$(\M(\Delta^n), \M(\Delta^{n-1}))$ splits, and
$$
\tilde{H}_k(\M(\Delta^n), \M(\Delta^{n-1})) \cong \tilde{H}_k(\M(\Delta^n)) \oplus \tilde{H}_{k-1}(\M(\Delta^{n-1})).
$$
\end{cor}

In particular, the inductive splitting reduces the computation of $\tilde{H}_*(\M(\Delta^n))$ to the relative homology $\tilde{H}_*(\M(\Delta^n), \M(\Delta^{n-1}))$. Whether this relative homology admits a tractable closed-form description is an open problem.

Although the framework developed in this section relates $\M(\Delta^n)$ to its subcomplexes, it does not yet determine the homotopy type of $\M(\Delta^n)$ for any specific $n \geq 3$. In the next section we carry out this determination in the first nontrivial case, combining the connectivity bound of Theorem \ref{thm: d-2 connected} with direct homology computations to obtain the homotopy types of $\M(\Delta^3)$, $\M(\partial\Delta^3)$, and $\M_{P}(\Delta^3)$.

\section{Homotopy types for $n = 3$ and computational data}\label{sec:n3}

We now determine the homotopy type of $\M(\Delta^3)$, and in the process those of $\M(\partial\Delta^3)$ and $\M_{P}(\Delta^3)$ as well. The strategy in each case is the same: we use the connectivity bound of Theorem \ref{thm: d-2 connected} to establish simple connectivity, compute the integral homology directly from the boundary matrices, and apply one of the Bravo-Camarena theorems to identify the homotopy type with a wedge of spheres.

\subsection{Computational data and homotopy types}

We begin by recording the computational data on which the arguments depend. The $f$-vectors of $\M(\Delta^3)$, $\M(\partial\Delta^3)$,
$\M_{P}(\Delta^3)$, and $\M_{P}(\partial\Delta^3)$ were computed by exhaustive enumeration of the acyclic matchings on the corresponding Hasse diagrams. From these $f$-vectors we constructed the integer boundary matrices and computed their Smith normal forms to obtain the integral homology groups; see the AI Acknowledgments for details on the software used. The $f$-vectors and Euler characteristics appear in Table \ref{tab:n3data} while the integral homology is shown in Table \ref{tab:n3homology}.

\begin{table}[h]
\centering
\begin{tabular}{l l l}
\toprule
Complex & $f$-vector & $\chi$ \\
\midrule
$\M(\Delta^3)$ & $(28, 300, 1544, 3932, 4632, 2128, 256)$ & $100$ \\
$\M(\partial\Delta^3)$ & $(24, 216, 896, 1692, 1248, 256)$ & $4$ \\
$\M_{P}(\Delta^3)$ & $(28, 300, 1544, 3680, 3672, 1600, 256)$ & $-80$ \\
$\M_{P}(\partial\Delta^3)$ & $(24, 216, 896, 1680, 1152, 256)$ & $-80$ \\
\bottomrule
\end{tabular}
\caption{$f$-vectors and Euler characteristics for $n = 3$.}
\label{tab:n3data}
\end{table}

\begin{table}[h]
\centering
\begin{tabular}{l l}
\toprule
Complex & Integral reduced homology \\
\midrule
$\M(\Delta^3)$ & $\tilde{H}_4 \cong \ZZ^{99}$, all other $\tilde{H}_k = 0$ \\
$\M(\partial\Delta^3)$ & $\tilde{H}_3 \cong \ZZ^{21}$, $\tilde{H}_4 \cong \ZZ^{24}$, all others $= 0$ \\
$\M_{P}(\Delta^3)$ & $\tilde{H}_3 \cong \ZZ^{81}$, all other $\tilde{H}_k = 0$ \\
$\M_{P}(\partial\Delta^3)$ & $\tilde{H}_3 \cong \ZZ^{81}$, all other $\tilde{H}_k = 0$ \\
\bottomrule
\end{tabular}
\caption{Integral homology for $n = 3$. All groups are torsion-free.}
\label{tab:n3homology}
\end{table}

All four complexes have torsion-free homology concentrated in at most two consecutive degrees. Combined with the connectivity bound of Theorem \ref{thm: d-2 connected}, this is enough to determine the homotopy types via the homological Bravo-Camarena theorems. We now state and prove each case in turn.

\begin{thm}\label{thm: h.t. of delta3}  
$\M(\Delta^3) \simeq \bigvee^{99} S^4$.
\end{thm}

\begin{proof}
Let $K = \Delta^3$. Since $K_{(1)}$ is $3$-regular, Theorem \ref{thm: d-2 connected} gives that $\M(K)$ is $1$-connected. From Table \ref{tab:n3homology}, $\tilde{H}_4(\M(K)) \cong \ZZ^{99}$ and $\tilde{H}_i(\M(K)) = 0$ for all
$i \neq 4$ (a computation also recorded in \cite[p. 47]{CJ-2005}). By Theorem \ref{thm: wedge of spheres homotopy}, $\M(\Delta^3) \simeq \bigvee^{99} S^4$.
\end{proof}

The same pipeline applies to $\M(\partial\Delta^3)$, but here the homology is concentrated in two consecutive degrees rather than one, so we invoke the double-wedge version, Theorem \ref{thm: double wedge of spheres homotopy}.

\begin{thm}\label{thm: h.t. of partial delta3}
$\M(\partial\Delta^3) \simeq \bigvee^{21} S^3 \vee \bigvee^{24} S^4$.
\end{thm}

\begin{proof}
Let $K = \partial\Delta^3$. Since $K_{(1)}$ is $3$-regular, Theorem \ref{thm: d-2 connected} gives that $\M(K)$ is $1$-connected, hence
simply connected. From Table \ref{tab:n3homology}, we have
$$
\tilde{H}_q(\M(\partial\Delta^3)) =
\begin{cases}
\ZZ^{21} & \text{for } q = 3, \\
\ZZ^{24} & \text{for } q = 4, \\
0 & \text{otherwise.}
\end{cases}
$$
By Theorem \ref{thm: double wedge of spheres homotopy},
$$
\M(\partial\Delta^3) \simeq \bigvee^{21} S^3 \vee \bigvee^{24} S^4.
$$
\end{proof}

For the pure complexes, we use the fact that the $2$-skeleton of $\M_{P}(K)$ coincides with that of $\M(K)$, so simple connectivity
of $\M(K)$ immediately implies the simple connectivity of $\M_{P}(K)$.

\begin{thm}\label{thm: h.t. of pure}
$\M_{P}(\Delta^3) \simeq \M_{P}(\partial\Delta^3)
\simeq \bigvee^{81} S^3$.
\end{thm}

\begin{proof}
Let $K$ denote either $\Delta^3$ or $\partial\Delta^3$. Direct comparison of the $f$-vectors in Table \ref{tab:n3data} shows that $\M(K)$ and $\M_{P}(K)$ have the same number of simplices in dimensions $0$, $1$, and $2$. Since $\M_{P}(K) \subseteq \M(K)$, the $2$-skeletons
of $\M_{P}(K)$ and $\M(K)$ coincide. The fundamental group of a CW complex is determined by its $2$-skeleton, so $\pi_1(\M_{P}(K)) \cong \pi_1(\M(K)) = 0$ by Theorem \ref{thm: d-2 connected}. From Table \ref{tab:n3homology},
$$
\tilde{H}_q(\M_{P}(K)) =
\begin{cases}
\ZZ^{81} & \text{for } q = 3, \\
0 & \text{otherwise.}
\end{cases}
$$
By Theorem \ref{thm: wedge of spheres homotopy}, $\M_{P}(K) \simeq \bigvee^{81} S^3$ in both cases.
\end{proof}

The fact that $\M_{P}(\Delta^3)$ and $\M_{P}(\partial\Delta^3)$ have the same homotopy type despite having very different $f$-vectors suggests that the phenomenon holds in general.

\begin{conj}\label{conj: pure equivalence}
For all $n \geq 2$, $\M_{P}(\Delta^n) \simeq \M_{P}(\partial\Delta^n)$.
\end{conj}

The conjecture holds for $n = 2$ trivially (both sides equal the corresponding non-pure complex by Example \ref{ex: delta2}(c)) and for $n = 3$ by Theorem \ref{thm: h.t. of pure}. For $n \geq 4$ it is open.

\subsection{Computational data for $n = 4$}\label{sec:data}

The complex $\M(\Delta^n)$ grows so rapidly with $n$ that direct computation quickly becomes infeasible. For $n = 3$, $\M(\Delta^3)$ has $12{,}820$ simplices and its homology is computable on a laptop. For $n = 4$, the $f$-vector alone required parallel enumeration on a computing cluster, and the homology is out of reach, as the largest boundary matrix has dimensions approximately $8.76 \times 10^8$ by $7.13 \times 10^8$. For $n = 5$, even the $f$-vector is computationally infeasible. This rapid growth underscores the necessity of the structural results developed in the preceding sections.

We computed the $f$-vector of $\M(\Delta^4)$ by exhaustive enumeration of acyclic matchings on the Hasse diagram of $\Delta^4$. The $f$-vector is
\begin{multline*}
(75,\ 2485,\ 47955,\ 598425,\ 5071367,\ 29844505,\ 122685075, \\
350017175,\ 680808105,\ 876110235,\ 712961065, \\
343320335,\ 88467825,\ 10315975,\ 380125),
\end{multline*}
with Euler characteristic $\chi(\M(\Delta^4)) = 212{,}457$, equivalently reduced Euler characteristic $\widetilde{\chi}(\M(\Delta^4)) = 212{,}456$. The number of top-dimensional facets is $f(4) = 380{,}125$, and the identity $380{,}125 = 5 \times 76{,}025$ confirms Theorem \ref{thm: morse bijection} with $|\text{top-dimensional facets of } \M(\Delta^4_{(2)})| = 76{,}025$.

The upper bound $f(n) \leq (n+1)^{2^{n-1}}$ of Chari and Joswig specializes to $f(4) \leq 5^8 = 390{,}625$. Our computation shows this bound is not tight.

For $n \leq 3$, $\M(\Delta^n)$ has integral homology concentrated in a single degree and is homotopy equivalent to a wedge of spheres of that degree, as $\M(\Delta^2) \simeq \bigvee^4 S^1$ and $\M(\Delta^3) \simeq \bigvee^{99} S^4$. One might hope that $\M(\Delta^n)$ is always a wedge of spheres of a single dimension $d(n)$, with $d(2) = 1$ and $d(3) = 4$. The simplest pattern consistent with these values is $d(n) = (n-1)^2$, which would predict $d(4) = 9$.

The Euler characteristic, however, rules this out. By Theorem \ref{thm: d-2 connected}, $\M(\Delta^4)$ is $2$-connected, so any
single-degree wedge decomposition $\M(\Delta^4) \simeq \bigvee^N S^d$ would require $d \geq 3$, with $\widetilde{\chi} = (-1)^d N$. Since
$\widetilde{\chi}(\M(\Delta^4)) = 212{,}456 > 0$, the degree $d$ must be even, and hence $d \geq 4$. In particular, $d = 9$ (and in general any odd degree) is impossible. The simplest single-degree possibility consistent with the data is $\M(\Delta^4) \simeq \bigvee^{212{,}456} S^4$, but the homology might equally well be spread across multiple even degrees.

Determining the actual homology groups of $\M(\Delta^4)$ would require computing the rank of sparse integer matrices with hundreds of millions of rows and columns, which is beyond current computational reach.

\section{The complex of generalized discrete Morse matchings}\label{sec: GM}

The generalized complex of discrete Morse matchings $\GM(K)$ was introduced in \cite{scoville2020higher} to aid in studying the connectivity of $\M(K)$. Unlike $\M(K)$, $\GM(K)$ allows matchings whose modified Hasse diagrams contain cycles; equivalently, $\GM(K)$ is the simplicial complex of all matchings on the Hasse diagram of $K$ with no acyclicity requirement. In particular, $\M(K) \subseteq \GM(K)$, and the difference between the two encodes the topological cost of the acyclicity condition.

A central question we now wish to consider is: how visible is the distinction between $\Delta^n$ and $\partial\Delta^n$ at the level of
these complexes? For $\M$, the answer is that the distinction is significant; explicitly, by Theorems \ref{thm: h.t. of delta3} and \ref{thm: h.t. of partial delta3}, $\M(\Delta^3) \not\simeq \M(\partial\Delta^3)$. The main observation of this section is that for $\GM$, the answer appears to be the opposite: at least for $n = 3$, the homotopy types coincide. We prove this directly and explore what happens when one tries to extend the argument to general $n$, where a natural strategy via vertex removal breaks down already at $n = 3$, leaving the general case open.

\subsection{Computational results for $n = 3$}\label{subsec: GM data}

\begin{thm}\label{thm: GM n3}
$\GM(\Delta^3) \simeq \GM(\partial\Delta^3) \simeq \bigvee^{39} S^4$.
\end{thm}

\begin{proof}
Observe that any two disjoint matched pairs of simplices on the modified Hasse diagram form a matching of length two which cannot contain a directed cycle, as two distinct simplices in a simplicial complex can share at most one face of codimension 1. Hence the $1$-skeletons of $\M(K)$ and $\GM(K)$ coincide for any $K$. Combined with cellular approximation, the inclusion $\M(K) \hookrightarrow \GM(K)$ induces a surjection $\pi_1(\M(K)) \to \pi_1(\GM(K))$. In particular, if $\M(K)$ is simply connected, so is $\GM(K)$.

By Theorem \ref{thm: d-2 connected}, both $\M(\Delta^3)$ and $\M(\partial\Delta^3)$ are simply connected, so $\GM(\Delta^3)$ and
$\GM(\partial\Delta^3)$ are as well. Direct computation of the homology from the boundary matrices yields
$$
\tilde{H}_4(\GM(\Delta^3)) \cong \tilde{H}_4(\GM(\partial\Delta^3)) \cong \ZZ^{39},
$$
with all other reduced homology groups vanishing. By Theorem \ref{thm: wedge of spheres homotopy}, $\GM(\Delta^3) \simeq \GM(\partial\Delta^3) \simeq \bigvee^{39} S^4$.
\end{proof}

The inclusion $\M(\Delta^3) \hookrightarrow \GM(\Delta^3)$ induces a surjection $\ZZ^{99} \to \ZZ^{39}$ on $H_4$. The $60$ classes
that are killed reflect the topological role of the acyclicity condition. These are $4$-dimensional homology cycles of acyclic matchings that bound in $\GM(\Delta^3)$ since the larger complex contains non-acyclic partial matchings that fill them in.

\subsection{Structural decomposition}\label{subsec: GM structure}

We now develop a structural decomposition of $\GM(\Delta^n)$ that mirrors the top-facet bijection for $\M$ of Proposition \ref{prop: top dim bijection}. Let $\sigma$ denote the unique $n$-simplex of $\Delta^n$, and for $i = 0, \ldots, n$ let $e_i = (F_i, \sigma)$ be the Hasse diagram edge which pairs the $i$-th facet with $\sigma$. Since the $e_i$ all share the simplex $\sigma$, at most one $e_i$ can appear in any partial matching, so the vertices $e_0, \ldots, e_n$ of $\GM(\Delta^n)$ have pairwise disjoint stars.

\begin{prop}\label{prop: GM decomposition}
There is a covering
$$
\GM(\Delta^n) = \GM(\partial\Delta^n) \cup \bigcup_{i=0}^{n} (e_i * L_i),
$$
where $L_i \subseteq \GM(\partial\Delta^n)$ is the subcomplex consisting of all matchings on $\partial\Delta^n$ that do not use any Hasse edge incident to $F_i$.
\end{prop}

\begin{proof}
Let $\mu$ be a simplex of $\GM(\Delta^n)$. If $\mu$ contains no $e_i$, then $\mu$ uses only Hasse edges of $\partial\Delta^n$, so $\mu \in \GM(\partial\Delta^n)$.

Otherwise, $\mu$ contains exactly one edge $e_i$. Write $\mu = \nu \cup \{e_i\}$. Since $e_i$ matches $F_i$ with $\sigma$, no Hasse edge of $\nu$ is incident to $F_i$, and since $\sigma \notin \partial\Delta^n$, all edges of $\nu$ lie in $\partial\Delta^n$. Thus $\nu \in L_i$, and $\mu \in e_i * L_i$.

Conversely, for any $\nu \in L_i$, the union $\{e_i\} \cup \nu$ is a valid partial matching on $\Delta^n$ and so lies in $\GM(\Delta^n)$.
\end{proof}

The covering is not disjoint. Indeed, each cone $e_i * L_i$ intersects $\GM(\partial\Delta^n)$ in $L_i$, and distinct cones overlap in $L_i \cap L_j$. Nonetheless, the decomposition suggests a strategy for the conjecture $\GM(\Delta^n) \simeq \GM(\partial\Delta^n)$: collapse each cone $e_i * L_i$ back to its base $L_i$, removing the apex $e_i$ in the process. 

\begin{cor}\label{cor: GM reduction}
If $L_i$ is collapsible for each $i = 0, \ldots, n$, then $\GM(\Delta^n) \simeq \GM(\partial\Delta^n)$.
\end{cor}

\begin{proof}
By \cite[Lemma 4.2.10]{BarmakThesis}, if the link of a vertex $v$ in a simplicial complex $K$ is collapsible, then $K$ collapses to the subcomplex obtained by removing every simplex containing $v$. We apply this to the vertices $e_0, \ldots, e_n$ of $\GM(\Delta^n)$. The link of $e_i$ in $\GM(\Delta^n)$ is precisely $L_i$, by Proposition \ref{prop: GM decomposition}. Since the $e_i$ have pairwise disjoint stars, removing them sequentially does not affect the links of the others, so if all $L_i$ are collapsible, the $e_i$ may be removed one by one without changing the homotopy type, collapsing down to $\GM(\partial\Delta^n)$.
\end{proof}

The natural question is now whether each $L_i$ is in fact collapsible. We can identify $L_i$ explicitly. The Hasse diagram edges of $\partial\Delta^n$ that avoid $F_i$ are precisely the edges of the Hasse diagram of the subcomplex of $\partial\Delta^n$ obtained by deleting the facet $F_i$. Since $F_i$ is the unique facet of $\partial\Delta^n$ not containing the vertex $v_i$ opposite to it, this subcomplex equals the star $\st_{\partial\Delta^n}(v_i) = v_i * \lk_{\partial\Delta^n}(v_i)
= v_i * \partial\Delta^{n-1}$. Hence
$$
L_i \cong \GM(v_i * \partial\Delta^{n-1}),
$$
and by symmetry all $L_i$ are isomorphic.

For $n = 2$, $\partial\Delta^1$ is two points, so $v_i * \partial\Delta^1$ is a path of length $2$, and $L_i = \GM(v_i * \partial\Delta^1)$ is a tree with $f$-vector $(4, 3)$, hence collapsible. By Corollary \ref{cor: GM reduction}, $\GM(\Delta^2) \simeq \GM(\partial\Delta^2)$.

For $n = 3$, the situation changes dramatically.

\begin{prop}\label{prop: L_i not collapsible}
For $n = 3$, each $L_i$ has $f$-vector $(21, 162, 570, 924, 612, 116)$ and Euler characteristic $\chi(L_i) = 1$ while the integral reduced homology of $L_i$ is given by
$$
\tilde{H}_3(L_i) \cong \ZZ^2, \quad \tilde{H}_4(L_i) \cong \ZZ^2,
$$
with all other reduced homology vanishing. In particular, $L_i$ is not collapsible, and the hypothesis of Corollary \ref{cor: GM reduction} fails for $n = 3$.
\end{prop}

The Euler characteristic check $\chi = 1$ is consistent with $\tilde\chi = 0$ and the cancellation $-2 + 2 = 0$ in the two nontrivial homology groups.

The failure of Corollary \ref{cor: GM reduction} at $n = 3$ is interesting because the homotopy equivalence $\GM(\Delta^3) \simeq \GM(\partial\Delta^3)$ nonetheless holds by Theorem \ref{thm: GM n3}. The cones $e_i * L_i$ are each individually contractible, but they are glued to $\GM(\partial\Delta^3)$ along $L_i$, which has nontrivial homotopy. Each such gluing alters the homotopy type of $\GM(\partial\Delta^3)$, and the alterations from all $n+1$ cones must somehow cancel out and return the original type.

\begin{conj}\label{conj:GM}
For all $n \geq 2$, $\GM(\Delta^n) \simeq \GM(\partial\Delta^n)$.
\end{conj}

By contrast, the analogous statement for $\M$ is false. By Theorems \ref{thm: h.t. of delta3} and \ref{thm: h.t. of partial delta3},
$$
\M(\Delta^3) \simeq \bigvee^{99} S^4 \quad\text{and}\quad \M(\partial\Delta^3) \simeq \bigvee^{21} S^3 \vee \bigvee^{24} S^4,
$$
so the acyclicity condition genuinely distinguishes $\Delta^n$ from $\partial\Delta^n$ in a way that the matching condition does
not. Understanding this asymmetry will require an approach beyond the vertex-removal strategy of Corollary \ref{cor: GM reduction}.

\section{Concluding remarks}\label{sec: conclusion}

The results of this paper give a partial picture of the homotopy types of $\M(\Delta^n)$, $\M(\partial\Delta^n)$, $\M_{P}(\Delta^n)$, and $\GM(\Delta^n)$. For $n = 3$, the homotopy types of all four complexes are determined explicitly. For general $n$, the structural results constrain these homotopy types but do not determine them.

Two open conjectures emerge naturally. The first, Conjecture \ref{conj: pure equivalence}, asserts that $\M_{P}(\Delta^n) \simeq \M_{P}(\partial\Delta^n)$ for all $n \geq 2$. The second, Conjecture \ref{conj:GM}, asserts the analogous statement for $\GM$. Both are verified for $n = 2$ and $n = 3$, and both fail for the non-pure $\M$ at $n = 3$, so the phenomena they describe are genuinely about the pure or generalized settings rather than artifacts of low dimensions.

The Euler characteristic computation $\widetilde{\chi}(\M(\Delta^4)) = 212{,}456$ rules out the simplest extrapolation from the $n = 3$ case, as $\M(\Delta^n)$ need not be a wedge of spheres in a single odd dimension. However, it leaves open whether $\M(\Delta^4)$ is itself a wedge of spheres in a single even dimension, and what its actual homology is. Resolving this would either require a substantial computational advance or a structural identification of the homotopy type.

Finally, the failure of the vertex-removal strategy for $\GM$ at $n = 3$ indicates that the equivalence $\GM(\Delta^n) \simeq \GM(\partial\Delta^n)$, if true, must arise from some subtler cancellation among the cones $e_i * L_i$. Identifying this cancellation may reveal new structure in the relationship between acyclicity and matchings on Hasse diagrams of simplices.

\section*{AI Acknowledgments}

The author acknowledges use of Anthropic's Claude Opus 4.6 for assistance throughout this project. Claude served as a dialogue partner during the development of proofs, flagged gaps and errors in arguments, helped structure and edit the exposition, and was instrumental in writing and validating the enumeration scripts, Smith normal form computations, and the paper's verification software. The author retains full responsibility for the correctness of all mathematical content; every proof and computation has been independently checked, and the code accompanying this paper (\url{https://github.com/nscoville/morse-complex-verification}) independently reproduces every numerical claim.

\bibliographystyle{amsplain}
\bibliography{MorseComplex}

\end{document}